\documentclass{amsart}
\usepackage{amsmath}
\usepackage{amsthm}
\usepackage{amssymb}
\usepackage{bm}
\usepackage{cite}
\usepackage{stmaryrd}
\usepackage{amsfonts} 
\usepackage{MnSymbol}
\usepackage{enumitem}
\DeclareMathSymbol{\sm}{\mathbin}{AMSa}{"39}

\usepackage{tikz}
\usetikzlibrary{cd}
\usepackage{mathtools}

\newtheorem{thm}{Theorem}
\newtheorem{lem}{Lemma}
\newtheorem{prop}{Proposition}
\newtheorem{cor}{Corollary}

\theoremstyle{definition}
\newtheorem{defn}{Definition}
\newtheorem{exam}{Example}
\theoremstyle{remark}
\newtheorem{rem}{Remark}

\title{Some Morse-type inequalities for symplectic manifolds}
\author{Thomas Machon}
\address{H.H.~Wills Physics Laboratory, University of Bristol, Tyndall Avenue, Bristol BS8 1TL, UK}
\email{t.machon@bristol.ac.uk}
\begin{document}

\maketitle

\begin{abstract}
Morse-type inequalities are given for the symplectic versions of the Bott-Chern and Aeppli cohomology groups defined by Tseng and Yau.
\end{abstract}
\vspace{2ex}
The symplectic versions of the Bott-Chern and Aeppli cohomology groups, defined by Tseng and Yau~\cite{tseng2012cohomology}, are finite-dimensional cohomology groups given in terms of differential forms on a closed symplectic manifold $(M, \omega)$ of dimension $2n$. They are  denoted $H^k_{d+d^\Lambda}(M)$ and $H^k_{dd^\Lambda}(M)$, $1 \leq k \leq 2n$, and satisfy the following inequalities~\cite{angella2015inequalities} relating their dimensions to the Betti numbers of $M$, $b_k$,
\begin{equation} \label{eq:ineq}
{\rm dim}\, H^k_{d+d^\Lambda}(M) = {\rm dim} \, H^k_{dd^\Lambda} (M) \geq  b_k.
\end{equation}
The purpose of this paper is to prove a set of Morse-type inequalities for these groups.
\begin{thm} \label{thm:1}
Let $(M, \omega)$ be an $2n$-dimensional closed symplectic manifold and $f$ a Morse function with $m_i$ critical points of index $i$, then for $k \leq n$
\begin{align} \label{eq:thm1}
{\rm dim}\, H^k_{d+d^\Lambda}(M)  \leq &\sum_{i=0} m_{k-2i}, \\
{\rm dim}\, H^k_{dd^\Lambda}(M)  \leq &\sum_{i=0} m_{2n-k+2i}.
\end{align}
\end{thm}
\begin{rem}
The dimensions of the $d+d^\Lambda$ and $dd^\Lambda$ groups can vary under homotopy of the symplectic form~\cite{tardini2019symplectic}.
\end{rem}
$(M, \omega)$ satisfying the hard Lefschetz condition is equivalent to the inequality \eqref{eq:ineq} being saturated for all $k$~\cite{tseng2012cohomology}, and in this case Theorem~\ref{thm:1} is strictly weaker than the weak Morse inequalities if $n \geq 2$ (and equivalent if $n=1$). If $(M, \omega)$ does not satisfy the hard Lefschetz condition, Theorem~\ref{thm:1} can give constraints on the $m_i$ additional to those obtained from the weak Morse inequalities. In particular, a short calculation gives the following.
\begin{cor} \label{cor:1}
The total number of critical points of a Morse function on a closed symplectic manifold $(M, \omega)$ of dimension $2n$ satisfies the inequality
$$ h^{n-2}+2h^{n-1}+ h^n \leq \sum_{i=0}^{2n} m_i,$$
where $h^k = {\rm dim}\, H^k_{d+d^\Lambda}(M)$.
\end{cor}
The quantity $h^{n-2}+2h^{n-1}+ h^n$ in Corollary~\ref{cor:1} can be larger than the sum of the Betti numbers, as the following example shows.
\begin{exam}
Let $M= \Gamma  \backslash G$ be the compact 4-manifold with structure equations
$$ \begin{cases} de_1 = 0 \\
de_2 = 0 \\
de_3 = e_1 \wedge e_2 \\
de_4 = e_1 \wedge e_3
\end{cases}, $$
and pick the symplectic form $\omega = e_1 \wedge e_4 + e_2 \wedge e_3$. Then a short computation shows (see Ref.\cite{tardini2019symplectic}, Example 4.6)
$$h^2 = 4, \quad b_2 = 2.$$
Since $h^0=b_0=1$ and $h^1=b_1$ always,
$$h^0+2h^1+h^2 =1+ \sum_{i=0}^{4} b_i \leq \sum_{i=0}^{4} m_i.$$
This inequality cannot be saturated, since the strong Morse inequalities must hold. Hence in this case we can strengthen the statement to
$$2+ \sum_{i=0}^{4} b_i \leq \sum_{i=0}^{4} m_i.$$
\end{exam}
Another example is given by the symplectic 4-manifold constructed by McMullen and Taubes~\cite{mcmullen19994}. The computations of the $d+d^\Lambda$ cohomology are given in Ref.~\cite{tsai2016cohomology}, Section 4.4.2.
\begin{exam}
Let $\Sigma$ be a closed surface with orientation preserving diffeomorphism $\tau$ and let $\omega_\tau$ be a $\tau$ invariant symplectic form. Then construct the manifolds
$$Y_\tau = \Sigma \times_\tau S^1, \quad X = S^1 \times Y_\tau.$$
Let $\phi$ be a coordinate for the $S^1$ base of $Y_\tau$ and $\theta$ a coordinate for the first $S^1$ factor in $X$. Then $X$ admits the symplectic form
$$ \omega = d\theta \wedge d \phi + \omega_\tau.$$
Consider the induced map $\tau^\ast- {\bm I}$ acting on the first de Rham cohomology $H_d^1(\Sigma)$. Let $q+p$, $q \geq p$ be the dimension of ${\rm ker}(\tau^\ast - {\bm I})$ and $q-p$ the dimension of ${\rm ker}(\tau^\ast - {\bm I}) \cap {\rm im} (\tau^\ast - {\bm I})$. Then (Ref.\cite{tsai2016cohomology}, Proposition 4.14)
$$ b_1 = b_3 = h^1=h^3 = q+p+2,$$
$$ b_2 = 2q+2p+2, \quad h^2 = 3q+p+2.$$
Theorem~\ref{thm:1} then gives the inequality
$$ q-p-1 + \sum_{i=0}^4  b_i   \leq  \sum_{i=0}^4 m_i.$$
\end{exam}

The strong Arnold conjecture (see Ref.\cite{golovko2020variants} for a summary) states that the minimal number of critical points of a Morse function is a lower bound for the number of fixed points of a generic Hamiltonian diffeomorphism. The weak version, proved through Floer theory~\cite{floer1989witten,fukaya1999arnold,liu1998floer}, replaces the sum of Morse critical points with a sum of Betti numbers. In this context it is natural to conjecture that the number of fixed points of a generic Hamiltonian diffeomorphism is bounded from below by
$$ h^{n-2}+2h^{n-1}+h^n.$$
In the other direction one can use knowledge about Morse functions on $M$ to constrain the behaviour of the $d+d^\Lambda$ and $dd^\Lambda$ cohomologies. An example is the following.
\begin{cor}
Let $(M, \omega)$ be a simply-connected symplectic manifold of dimension $\geq 6$. Then for $k \leq n$,
\begin{align*} 
{\rm dim}\, H^k_{d+d^\Lambda}(M)  \leq &\sum_{i=0} b_{k-2i}+ 2 \tau_{k-2i}, \\
{\rm dim}\, H^k_{dd^\Lambda}(M)  \leq &\sum_{i=0} b_{2n-k+2i}+2 \tau_{2n-k+2i},
\end{align*}
where $\tau_i$ is the number of torsion generators in $H_i(M; \mathbb{Z})$.
\end{cor}
Upper bounds for the $d+d^\Lambda$ and $dd^\Lambda$ groups have been given by Angella and Tardini, who give inequalities bounding certain combinations of dimensions of the $H^k_{d+d^\Lambda}$ and $H^k_{d+d^\Lambda}$ groups in terms of Betti numbers\cite{angella2017quantitative,tardini2017cohomological}. Tseng and Yau also defined related cohomological invariants~\cite{tseng2012cohomology2} (see also Ref.\cite{tsai2016cohomology}) originating from an elliptic complex (also studied by Smith in four dimensions~\cite{smith1976examples}). When the de Rham class of $\omega$ is integral, we have the prequantum line bundle $L$ and associated sphere bundles $E_p = S(L^{\oplus(p+1)})$. In Ref.\cite{tanaka2018odd} the invariants associated to the elliptic complex were identified with the regular de Rham cohomology of the $E_p$, to which one may apply the regular Morse inequalities.

The proof of Theorem~\ref{thm:1} is analytic and similar in nature to Witten's construction for Morse theory~\cite{witten1982supersymmetry}. In Sections~\ref{sec:lef} and~\ref{sec:coh} we review material on the Lefschetz decomposition of differential forms as well as basic properties of the $d+d^\Lambda$ and $dd^\Lambda$ cohomology groups. Section~\ref{sec:hodge} discusses the Hodge theory of these groups, introducing an alternative elliptic operator characterising the harmonic $d+d^\Lambda$ forms. In Section~\ref{sec:witten} we then show that an appropriate Witten-Novikov deformation respects the Lefschetz decomposition of harmonic forms in the Hodge theory of the $d+d^\Lambda$ cohomology. Section~\ref{sec:morse} discusses a compatibility condition between a Morse function and symplectic form, which is a local version of the condition found in the study of Weinstein manifolds~\cite{cieliebak2012stein}. Requiring this condition ensures that the localization phenomenon of harmonic forms associated to the Witten-Novikov deformation interacts well with the Lefschetz decomposition. Sections~\ref{sec:local}--{\ref{sec:proof} are then concerned with the proof of Theorem~\ref{thm:1}, following, {\it mutatis mutandis}, the proof of the Morse inequalities given by Zhang~\cite{zhang2001lectures}.

\section{The Lefschetz decomposition} \label{sec:lef}
Much of the material in Sections~\ref{sec:lef}-\ref{sec:hodge} is a summary of existing results, we refer the reader to Refs.~\cite{yan1996hodge,tseng2012cohomology} for more details. Throughout $(M, \omega)$ will be a closed symplectic manifold of dimension $2n$, and $(\omega, g, J)$ will be a compatible triple of symplectic form, metric, and almost complex structure respectively. The Lefschetz operator $L$ is defined on the space of differential $k$-forms, $\Omega^k(M)$
$$ L : \Omega^k(M) \to \Omega^{k+2}(M), \quad \alpha \mapsto \omega \wedge \alpha,$$
along with its dual operator
$$\Lambda : \Omega^k(M) \to \Omega^{k-2}(M), \quad \alpha \mapsto \iota_\pi \alpha,$$
where $\pi$ is the Poisson structure obtained as the inverse of $\omega$. Concretely, in local Darboux coordinates we define
$$\omega = \sum_{i=1}^n dx_{2i-1} \wedge dx_{2i}, \quad \pi = \sum_{i=1} \partial_{{2i-1}} \wedge \partial_{2i},$$
where $\partial_i$ is the coordinate vector field associated to $x_i$. If we further define the index counting operator $H$ acting on $k$ forms as
$$ H: \alpha \mapsto (n-k) \alpha,$$
then we obtain the $\mathfrak{sl}(2, \mathbb{R})$ representation~\cite{yan1996hodge}
$$[\Lambda, L] = H, \quad [\Lambda, H] = 2 \Lambda, \quad [L, H] = -2L.$$
This yields the Lefschetz decomposition of differential forms on $M$~\cite{weil1958introduction}
$$\Omega^k(M) = \bigoplus_{r={\rm max}(0,k-n)} L^r P\Omega^{k-2r}(M),$$
where $P\Omega^{k-2r}$ is the space of primitive $k$-forms, defined as $k$-forms $\alpha$ satisfying $\Lambda \alpha = 0$, or equivalently $L^{n-k+1} \alpha = 0$. There is a corresponding dual (or coeffective~\cite{bouche1990cohomologie}) decomposition in terms of coeffective (or coprimitive) forms, which are defined as differential $k$-forms satisfying $L \alpha = 0$, or equivalently $\Lambda^{n-k+1} \alpha=0$.

As has been observed~\cite{tseng2012cohomology,bahramgiri,angella2014symplectic,lin2011currents}, primitive and coeffective differential forms are closely related to coisotropic and isotropic subspaces respectively. We will elaborate on this relationship a little, as it motivates the notion of 
compatible Morse function in Section~\ref{sec:morse}. Suppose $W$ is a symplectic space of dimension $2n$ and let $V$ be a $k$-dimensional subspace of $W$ which is isotropic if $k<n$, Lagrangian if $k=n$ and coistropic if $k>n$. Then on can associate to $V$ is an element of the exterior algebra $\alpha_V \in \bigwedge^{2n-k} W$ defined as the contraction of a basis of $V$ with an arbitrary volume element in $\bigwedge^{2n} W$. The following is easily seen.
\begin{prop}
If $V$ is isotropic, $\alpha_V$ is coeffective. If $V$ is coisotropic, $\alpha_V$ is primitive. If $V$ is Lagrangian, $\alpha_V$ is both coisotropic and coeffective.
\end{prop}
This has the immediate consequence that, for example, the dual current of coisotropic submanifold is primitive (see Refs.\cite{tseng2012cohomology, bahramgiri,angella2014symplectic,lin2011currents}).
\begin{rem}
There are many examples of primitive or coeffective forms that do not correspond to coisotropic or isotropic spaces (or indeed formal linear combinations thereof). For example, on $\mathbb{R}^4$ with the symplectic form $e_1 \wedge e_2 + e_3 \wedge e_4$, the form $e_1 \wedge e_2 - e_3 \wedge e_4$ is primitive, but does not correspond to any Lagrangian submanifold. This observation is the basis for the proof of Lemma~\ref{lem:bigo}.
\end{rem}

\section{The $d+d^\Lambda$ and $dd^\Lambda$ cohomologies} \label{sec:coh}

On a symplectic manifold one may define the symplectic differential~\cite{koszul1985crochet, brylinski1988differential} $d^\Lambda$ on $k$-forms
$$ d^\Lambda : \Omega^k(M) \to \Omega^{k-1}(M), \quad d^\Lambda: \alpha \mapsto [d, \Lambda] \alpha$$
The property $(d^\Lambda)^2=0$ follows from the integrability of the Poisson structure dual to $\omega$. The associated cohomology groups
$$ H^k_{d^\Lambda}(M) = \frac{{\rm ker} d^\Lambda \cap \Omega^k(M)}{d^\Lambda \Omega^{k+1}(M)},$$
are isomorphic to the de Rham groups, $H^k_{d}(M) \cong H^k_{d^\Lambda}(M)$ (see Ref.\cite{brylinski1988differential}). The differentials $d$ and $d^\Lambda$ interact with the ${\mathfrak{sl}}(2, \mathbb{R})$ representation in the following way (Ref.\cite{tseng2012cohomology}, Lemma 2.3).
\begin{lem} \label{lem:rel1}
The differentials $d$ and $d^\Lambda$ satisfy the following commutation relations with respect to the $\mathfrak{sl}(2, \mathbb{R})$ representation $(L, \Lambda, H)$.
\begin{align*}
&[d, L] = 0, & \quad & [d, \Lambda]  = d^\Lambda,  & \quad &  [d,H] = d,\\
&[d^\Lambda, L] = d , & \quad &   [d^\Lambda, \Lambda]  = 0 , & \quad & [d^\Lambda,H] = -d^\Lambda, \\
&[d d^\Lambda, L] = 0 , & \quad &  [d d^\Lambda, \Lambda]  = 0, & \quad &  [d d^\Lambda,H] = 0.
\end{align*}
\end{lem}
These differentials are used to construct symplectic analogues of the Bott-Chern and Aeppli cohomologies for complex manifolds. The cohomology group $H^k_{d+d^\Lambda}(M)$ is defined as the cohomology of the following short differential complex
\begin{center}
\begin{tikzcd}
\Omega^k(M) \arrow[r,"d d^\Lambda"]  & \Omega^k(M) \arrow[r,"d+d^\Lambda"]  & \Omega^{k+1}(M) \oplus \Omega^{k-1}(M)\end{tikzcd}
\end{center}
so that
$$H^k_{d+d^\Lambda}(M) = \frac{{\rm ker} \, (d+d^\Lambda) \cap \Omega^k(M)}{{\rm im}\, dd^\Lambda \cap \Omega^k(M)}.$$
The $dd^\Lambda$ cohomology is similarly defined using the complex
\begin{center}
\begin{tikzcd}
\Omega^{k-1}(M) \arrow[rd,"d"] & &  \\
& \Omega^k(M) \arrow[r,"dd^\Lambda"]  & \Omega^k(M) \\
\Omega^{k+1}(M) \arrow[ru,"d^\Lambda"]& & \end{tikzcd}
\end{center}
so that
$$ H^k_{dd^\Lambda}(M) = \frac{{\rm ker}\, dd^\Lambda \cap \Omega^k(M)}{({\rm im}\ d+ {\rm im} d^\Lambda) \cap \Omega^k(M)}.$$
Both the $d+d^\Lambda$ and $dd^\Lambda$ groups are finite dimensional, and there are isomorphisms
$$ H^k_{d+d^\Lambda} \cong H^{2n-k}_{d+d^\Lambda} \cong H^k_{dd^\Lambda}.$$

\section{Hodge theory for the $d+d^\Lambda$ cohomology} \label{sec:hodge}

The Riemannian metric $g$ gives the associated Hodge star $\ast$, which allows us to define the inner product on forms
$$ (\alpha, \beta) = \int_M \alpha \wedge \ast \beta,$$
and the norm
$$ \| \alpha \|^2 = \int_M \alpha \wedge\ast \alpha.$$
It follows from the compatibility of the triple $(\omega, g , J)$ (see Ref.\cite{tseng2012cohomology}) that $\Lambda$ is the adjoint of $L$ with respect to the pairing $(\cdot, \cdot )$, so that
$$(L \alpha, \beta) = (\alpha, \Lambda \beta).$$
We may the define the codifferential $d^\ast$ and the symplectic codifferential $d^{\Lambda \ast}$ as the adjoints of $d$ and $d^\Lambda$ respectively,
$$ d^\ast = - \ast d \ast, \quad d^{\Lambda \ast} = [L, d^\ast].$$
We then have the following commutation relations (see Ref.\cite{tseng2012cohomology}, Lemma 2.10).
\begin{lem}
The differential operators $(d^\ast, d^{\Lambda \ast}, d^\ast d^{\Lambda \ast})$ satisfy the following commutation relations with respect to the $\mathfrak{sl}(2, \mathbb{R})$ representation $(L, \Lambda, H)$.
\begin{align*}
&[d^\ast, L] = -d^{\Lambda \ast}, &\quad & [d^\ast, \Lambda]  = 0, & \quad &  [d^\ast,H] = - d^\ast, \\
&[d^{\Lambda \ast}, L] = 0 ,& \quad &  [d^{\Lambda \ast}, \Lambda]  = -d^\ast , &\quad & [d^{\Lambda \ast},H] = d^{\Lambda \ast}, \\
&[d^\ast d^{\Lambda \ast}, L] = 0 , &\quad&  [d^\ast d^{\Lambda \ast}, \Lambda] = 0 , &\quad &  [d^\ast d^{\Lambda \ast},H] = 0.
\end{align*}
\end{lem}

The short differential complex for the $d+d^\Lambda$ cohomology suggests the self-adjoint operator
\begin{equation} \label{eq:opbad}
d d^\Lambda d^{\Lambda \ast} d^\ast +\lambda (d^\ast d+ d^{\Lambda \ast} d^\Lambda),
\end{equation}
where $\lambda$ is an arbitrary constant, should act as a Laplacian for the $d+d^\Lambda$ cohomology. Consequently, one defines a differential form $\alpha$ as $d+d^\Lambda$ harmonic if it satisfies
$$ d\alpha = d^\Lambda \alpha = d^{\Lambda \ast} d^\ast \alpha = 0,$$
with $\mathcal{H}^k_{d+d^\Lambda}(M)$ the space of harmonic forms. The operator \eqref{eq:opbad} is not elliptic, so cannot be used as the basis for a Hodge theory of the $d+d^\Lambda$ cohomology. Instead Tseng and Yau~\cite{tseng2012cohomology}  introduce a fourth-order self-adjoint elliptic operator whose kernel corresponds precisely to the harmonic $d+d^\Lambda$ forms,
\begin{equation} \label{op:bad}
d d^\Lambda d^{\Lambda \ast} d^\ast + d^{\Lambda \ast} d^\ast d d^\Lambda + d^\ast d^\Lambda d^{\Lambda \ast} d+ d^{\Lambda \ast} d d^\ast d^\Lambda + \lambda (d^\ast d+ d^{\Lambda \ast} d^\Lambda).\end{equation}
This operator is modelled after the similar operator used in the harmonic theory of the Bott-Chern cohomology of complex manifolds (see Ref.\cite{kodaira201548} Proposition 5 and Ref.\cite{schweitzer2007autour}, Section 2.b). The operator \eqref{eq:opbad} has two inconvenient properties. Firstly, it does not commute with the operators $L$ and $\Lambda$, so does not immediately give the Lefschetz decomposition of harmonic forms. Secondly it depends on an arbitrary constant $\lambda$ of dimension $[{\rm Length}]^{-2}$, so that solving $D_{d+d^\Lambda}=0$ amounts to solving two separate equations (the fourth and quadratic order terms). Instead we introduce the fourth-order operator $\mathcal{D}: \Omega^k(M) \to \Omega^k(M)$ given by
\begin{equation} \label{eq:D} \mathcal{D} = d^\ast d d^\ast d + d^{\Lambda \ast }d^{\Lambda} d^{\Lambda \ast} d^\Lambda + d^{\Lambda \ast} d d^\ast d^\Lambda +  d^{\ast} d^\Lambda d^{\Lambda \ast} d + 2 d d^\Lambda d^{\Lambda \ast} d^\ast.
\end{equation}
\begin{prop} \label{prop:D}
The operator $\mathcal{D}$ has the following properties:
\begin{enumerate}
\item $\mathcal{D}$ is elliptic and self-adjoint,
\item ${\rm ker}\; \mathcal{D} = \mathcal{H}^k_{d+d^\Lambda}(M)$,
\item $[\Lambda, \mathcal{D}] = [L, \mathcal{D}]=0$.
\end{enumerate}
\end{prop}
\begin{proof}
$\mathcal{D}$ is evidently self-adjoint. To prove ellipticity we perform a symbol calculation, we use $\cong$ to denote equivalence of symbols. We use the relation (Ref.\cite{tseng2012cohomology} Proposition 3.3) 
\begin{equation} \label{eq:2sym}
d^\ast d + d d^\ast \cong d^\Lambda d^{\Lambda \ast} + d^{\Lambda \ast} d^\Lambda,
\end{equation}
as well as the relations (equalities on K\"{a}hler manifolds, see Ref.\cite{tseng2012cohomology} Theorem 3.5) $d^\ast d^\Lambda \cong - d^\ast d^\Lambda$, $d d^{\Lambda \ast} = - d^{\Lambda \ast} d$. Applying the operators $d d^\ast$ and $d^\Lambda d^{\Lambda \ast}$ to \eqref{eq:2sym} we obtain
\begin{equation} \label{eq:SYMB} d d^\ast d d^\ast \cong d d^\Lambda d^{\Lambda \ast} d^\ast + d^{\Lambda \ast} d d^\ast d^\Lambda, \quad d^\Lambda d^{\Lambda \ast} d^\Lambda d^{\Lambda \ast} \cong  d d^\Lambda d^{\Lambda \ast} d^\ast + d^{\Lambda \ast} d d^\ast d^\Lambda, 
\end{equation}
which gives 
$$ \mathcal{D} \cong d^\ast d d^\ast + dd^\ast d d^\ast + d^{\Lambda \ast }d^{\Lambda} d^{\Lambda \ast} d^\Lambda + d^{\Lambda  }d^{\Lambda \ast} d^{\Lambda } d^{\Lambda  \ast} \cong 2(d d^\ast + d^\ast d)^2, $$
and hence $\mathcal{D}$ is elliptic. To prove that the kernel of $\mathcal{D}$ corresponds to harmonic forms note that
$$ \mathcal{D} \alpha = 0 \; \Rightarrow \| d^\ast d \alpha \|^2 + \| d^{\Lambda \ast} d^\Lambda  \alpha \|^2 + \| d^\ast d^\Lambda \alpha \|^2 + \| d^{\Lambda \ast} d \alpha \|^2 + 2 \| d^{\Lambda \ast} d^\ast  \alpha \|^2 = 0 $$
which implies
$$ d \alpha = d^\Lambda \alpha = d^{\Lambda \ast} d^\ast \alpha = 0,$$
and hence the kernel of $\mathcal{D}$ is precisely the harmonic $d+d^\Lambda$ forms. Finally, to prove the commutation relations one uses the formulae of Lemma~\ref{lem:rel1}. First note that the final term of $\mathcal{D}$ commutes with both $\Lambda$ and $L$. The commutators of the first four terms of $\mathcal{D}$ can be computed explicitly,
\begin{align*}
&[L, d^\ast d d^\ast d] = d^{\Lambda \ast} d d^\ast d + d^\ast d d^{\Lambda \ast} d, & \quad &  [\Lambda, d^\ast d d^\ast d] = - d^\ast d^\Lambda d^\ast d - d^\ast d d^\ast d^\Lambda, & \\
&[L, d^{\Lambda \ast} d^{\Lambda} d^{\Lambda \ast} d^{\Lambda}] = - d^{\Lambda \ast} d d^{\Lambda \ast} d^{\Lambda}- d^{\Lambda \ast} d^{\Lambda} d^{\Lambda \ast} d, & \quad & [\Lambda, d^{\Lambda \ast} d^{\Lambda} d^{\Lambda \ast} d^{\Lambda}] = d^{\ast} d^{\Lambda} d^{\Lambda \ast} d^{\Lambda} + d^{\Lambda \ast} d^{\Lambda} d^{\ast} d^{\Lambda}, &\\
&[L, d^{\Lambda \ast} d d^\ast d^\Lambda ] = d^{\Lambda \ast} d d^{\Lambda \ast}  d^\Lambda - d^{\Lambda \ast} d d^\ast d, & \quad & [\Lambda, d^{\Lambda \ast} d d^\ast d^\Lambda ] =  d^{\ast} d d^\ast d^\Lambda-d^{\Lambda \ast} d^\Lambda d^\ast d^\Lambda, &\\
&[L, d^{\ast} d^\Lambda d^{\Lambda \ast} d] = d^{\Lambda \ast} d^\Lambda d^{\Lambda \ast} d -  d^{\ast} d d^{\Lambda \ast} d, & \quad &  [\Lambda, d^{\ast} d^\Lambda d^{\Lambda \ast} d ] = d^{\ast} d^\Lambda d^ {\ast} d - d^{\ast} d^\Lambda d^{\Lambda \ast} d^\Lambda,&
\end{align*}
%
%
%
the sum of these terms for $[L, \cdot]$ and $[\Lambda, \cdot]$ respectively vanish.
\end{proof}

As a consequence of the relation $[\mathcal{D}, \Lambda]=0$, the harmonic forms admit a Lefschetz decomposition (see Ref.\cite{tseng2012cohomology}, Theorem 3.22)
$$ \mathcal{H}^k_{d+d^\Lambda}(M) = \bigoplus_{r= {\rm max}(k-n,0)} L^r P\mathcal{H}_{d+d^\Lambda}^{k-2r}(M),$$
and we are led to study the primitive harmonic forms $P\mathcal{H}_{d+d^\Lambda}^{k}(M)$. The requirement $d\alpha = d^\Lambda \alpha=0$ is redundant if $\alpha$ is primitive, as in this case $d^\Lambda =-\Lambda d$, hence the primitive harmonic forms are given by solving
$$\Lambda \alpha = d\alpha = d^{\Lambda \ast} d^\ast \alpha = 0,$$
and are characterised by the following operator.
\begin{prop}
Let $\mathcal{D}_P : P\Omega^k(M) \to \Omega^k(M)$ be the differential operator given by
$$ \mathcal{D}_P =  d^\ast d d^\ast d  + d^{\Lambda \ast} d d^\ast d^\Lambda +  d d^\Lambda d^{\Lambda \ast} d^\ast.$$
Then:
\begin{enumerate}
\item[] $ \mathcal{D}_P$ is elliptic and self-adjoint,
\item[] $ {\rm ker} \; \mathcal{D}_P = P\mathcal{H}^k_{d+d^\Lambda}(M)$.
\end{enumerate}
\end{prop}
\begin{proof}
The first assertion follows immediately from half the symbol calculation in Proposition~\ref{prop:D}. The second follows from the requirement that
$$ \mathcal{D}_P \alpha = 0 \; \Rightarrow \| d^\ast d \alpha \|^2 + \| d^\ast d^\Lambda \alpha \|^2  +2 \| d^{\Lambda \ast} d^\ast \alpha \|^2 = 0 ,$$
which requires $d \alpha = 0$ and $d^{\Lambda \ast} d^\ast \alpha = 0$.
\end{proof}

\section{Witten-Novikov deformation of $d$ and $d^\Lambda$.} \label{sec:witten}

Given a closed 1-form $\alpha$, the Witten Novikov deformation~\cite{witten1982supersymmetry, novikov1982hamiltonian} of the exterior derivative is
$$ d \mapsto d+ \alpha \wedge .$$
In the case $\alpha= df$ is exact this can also be written as
\begin{equation} \label{eq:w1}
d \mapsto d_f = e^{-f} d e^{f} = d + df \wedge.
\end{equation}
There is a corresponding deformation of the symplectic differential $d^\Lambda = [d, \Lambda]$ given by
\begin{equation} \label{eq:w2}
d^\Lambda \mapsto d^\Lambda_f = e^{-f} d^\Lambda e^{f} = d^\Lambda - \iota_{V_f} = [d_F, \Lambda] ,
\end{equation}
where $V_f = \pi^\sharp(df)$ is the Hamiltonian vector field associated to $f$. 
\begin{rem}
The analogous operation of replacing $df$ with a closed 1-form $\alpha$ in \eqref{eq:w2} is to replace $V_f$ with a symplectic vector field. Note also that the differential $d^\Lambda$ may be defined for a general Poisson manifold (the Koszul-Brylinski differential~\cite{brylinski1988differential,koszul1985crochet}), in which case one may replace $V_f$ with an arbitrary Poisson vector field. Such a differential arises naturally~\cite{evens1999poisson,machon2020poisson} where $V_f$ is taken to be ($\pm$) the modular vector field of the Poisson structure~\cite{weinstein1997modular}.
\end{rem}

\begin{lem}
The groups $H^k_{d_f+d^\Lambda_f}$ and $H^k_{d_f d^\Lambda_f}$ do not depend on $f$.
\end{lem}
\begin{proof}
The map $\Omega^\bullet(M) \to e^{-f} \Omega^\bullet(M)$ induces an isomorphism between $H^k_{d_f+d^\Lambda_f}$ and $H^k_{d+d^\Lambda}$. Similarly for the $dd^\Lambda$ groups. 
\end{proof}
Because they arise through conjugation by $e^f$, the commutation relations of the deformed differentials with the operator $L$, $\Lambda$ and $H$ are preserved under the deformation by $f$.
\begin{lem}
The Witten-deformed differential operators $(d_f, d^\Lambda_f, d_f d^\Lambda_f)$ satisfy the following commutation relations with respect to the $\mathfrak{sl}(2, \mathbb{R})$ representation $(L, \Lambda, H)$.
\begin{align*}
&[d_f, L] = 0, & \quad & [d_f, \Lambda]  = d^\Lambda_f,  & \quad &  [d_f,H] = d_f,\\
&[d^\Lambda_f, L] = d_f , & \quad &   [d^\Lambda_f, \Lambda]  = 0 , & \quad & [d^\Lambda_f,H] = -d^\Lambda_f, \\
&[d_f d_f^\Lambda, L] = 0 , & \quad &  [d_f d_f^\Lambda, \Lambda]  = 0, & \quad &  [d_f d_f^\Lambda,H] = 0.
\end{align*}
\end{lem}
We also obtain the Witten deformation of the codifferential
$$ d_f^\ast = e^{f} d^\ast e^{-f} = d^\ast + \iota_{\nabla f},$$
as well as the deformation of the symplectic codifferential ${d^\Lambda}^\ast = - \ast d^\Lambda \ast =   [L, d^\ast]$
$$ {d^\Lambda_f}^\ast = e^{f} {d^\Lambda}^\ast e^{-f} =  {d^\Lambda}^\ast  + (\iota_{\nabla f} \omega) \wedge = [L, d_f^\ast].$$
We then, once again, immediately obtain the same deformed commutation relations as for the $d^\ast$ and $d^{\Lambda \ast}$ differentials.
\begin{lem}
The Witten-deformed differential operators $(d^\ast_f, {d^\Lambda_f}^\ast, {d_f d^\Lambda_f}^\ast)$ satisfy the following commutation relations with respect to the $\mathfrak{sl}(2, \mathbb{R})$ representation $(L, \Lambda, H)$.
\begin{align*}
&[d_f^\ast, L] = -d_f^{\Lambda \ast}, &\quad & [d_f^\ast, \Lambda]  = 0, & \quad &  [d_f^\ast,H] = - d_f^\ast, \\
&[d_f^{\Lambda \ast}, L] = 0 ,& \quad &  [d_f^{\Lambda \ast}, \Lambda]  = -d_f^\ast , &\quad & [d_f^{\Lambda \ast},H] = d_f^{\Lambda \ast}, \\
&[d_f^\ast d_f^{\Lambda \ast}, L] = 0 , &\quad&  [d_f^\ast d_f^{\Lambda \ast}, \Lambda] = 0 , &\quad &  [d_f^\ast d_f^{\Lambda \ast},H] = 0.
\end{align*}
\end{lem}
We then introduce the deformed operator $\mathcal{D}_f$ which is elliptic and self-adjoint.
$$ \mathcal{D}_f = d_f^\ast d_f d_f^\ast d_f + d_f^{\Lambda \ast }d_f^{\Lambda} d_f^{\Lambda \ast} d_f^\Lambda + d_f^{\Lambda \ast} d_f d_f^\ast d_f^\Lambda +  d_f^{\ast} d_f^\Lambda d_f^{\Lambda \ast} d_f + 2 d_f d_f^\Lambda d_f^{\Lambda \ast} d_f^\ast.$$
We can then introduce the Witten-deformed harmonic $d+d^\Lambda$ forms as the kernel of $\mathcal{D}_f$, which are solutions of the equations
$$ d_f \alpha = d^\Lambda_f \alpha = d_f^{\Lambda \ast} d_f^\ast \alpha = 0.$$
The properties of of these harmonic forms are summarised below.
\begin{prop} \label{prop:hd}
Let $\mathcal{H}^k_{d_f+d^\Lambda_f}(M)$ denote the space of $d_f+ d^\Lambda_f$ Harmonic $k$-forms. Then,
$${\rm dim}\,  \mathcal{H}^k_{d_f+d^\Lambda_f}(M) = {\rm dim}\,  \mathcal{H}^k_{d+d^\Lambda}(M)= {\rm dim}\,  H^k_{d+d^\Lambda}(M) < \infty, $$
and there is a Lefschetz decomposition
$$ \mathcal{H}^k_{d_f+d^\Lambda_f}(M)  = \bigoplus_{r} L^r P\mathcal{H}^{k-2r}_{d_f+d_f^\Lambda}(M).$$
\end{prop}
As in Section~\ref{sec:hodge}, we can introduce the differential operator
$$\mathcal{D}_{Pf} = \mathcal{D}_P =  d_f^\ast d_f d_f^\ast d_f  + d_f^{\Lambda \ast} d_f d_f^\ast d_f^\Lambda +  d_f d_f^\Lambda d_f^{\Lambda \ast} d_f^\ast,$$
whose kernel corresponds to the primitive $d_f+d^\Lambda_f$ harmonic forms,
$${\rm ker}\, \mathcal{D}_{Pf} \cap P\Omega^k(M) = P\mathcal{H}^{k-2r}_{d_f+d_f^\Lambda}(M).$$
There is an analogous theory for the $dd^\Lambda$ cohomology, but it is not required for our purposes.

\section{Compatible Morse functions} \label{sec:morse}
Our goal is to study the primitive $d_f+d_f^\Lambda$ harmonic forms, with the idea to localize them near critical points of $f$. In order for this to be tractable, we require the critical points satisfy a compatibility condition with the symplectic form. Our compatibility condition is, roughly, a local version of the condition on the Morse function in Weinstein manifolds~\cite{cieliebak2012stein}. Let $p \in M$ be a critical point of a Morse function with Morse index $n_p$. We say that $p$ is compatible with the symplectic form $\omega$ if there is a compatible triple $(\omega, g, J)$ such that $p$ satisfies the Morse-Smale condition and on an open neighbourhood $U_p$ of $p$ the unstable (stable) manifold of $p$ is: 
\begin{enumerate}
\item[] coisotropic (isotropic) if $n<n_p$,
\item[] Lagrangian (Lagrangian) if $n=n_p$,
\item[] isotropic (cisotropic) if  $n>n_p$.
\end{enumerate}
\begin{defn} \label{def:morse}
A Morse function $f$ is compatible with $\omega$ if there is a metric $g$, compatible with $\omega$, making all critical points of $f$ compatible.
\end{defn} 
\begin{rem}
This condition is satisfied by zeros of the gradient-like vector field on a Weinstein manifold (see Ref.~\cite{cieliebak2012stein}, Proposition 11.9).
\end{rem}
The following Lemma shows that any Morse function can be made compatible by diffeomorphism.\newpage
\begin{lem} \label{lem:what}
Let $f$ be a Morse function on a symplectic manifold $(M, \omega)$ of dimension $2n$. There is a diffeomorphism $\phi$ equal to the identity outside a neighbourhood of each critical point $p$ of $f$, such that $f$ is compatible with $\phi^\ast \omega$. Moreover, the diffeomorphism can be chosen such that in Darboux coordinates on a neighbourhood of $p$, $\omega$ and $f$ take the form
$$  \omega = \sum_{i=1}^n dx_{2i-1} \wedge dx_{2i} , \quad \phi^\ast f =f(0)+\frac{1}{2} \sum_{i=1}^{n_p} (-x_{2i-1}^2+x_{2i}^2 ) + \frac{1}{2}\sum_{i=n_p+1}^n (x_{2i-1}^2+x_{2i}^2),$$
if $n_p \leq n$ and
$$  \omega  = \sum_{i=1}^n dx_{2i-1} \wedge dx_{2i} , \quad \phi^\ast f = f(0)+ \frac{1}{2}\sum_{i=1}^{2n-n_p} (-x_{2i-1}^2+x_{2i}^2)  + \frac{1}{2}\sum_{i=2n_p-n+1}^n (- x_{2i-1}^2-x_{2i}^2),$$
if $ n_p>n$.
\end{lem}
\begin{proof}
Since the critical point is Morse, there is always diffeomorphism germ taking $f$ to the required local form, and we may choose this germ to be orientation preserving. This can then be smoothly extended to a diffeomorphism $\phi$ of $M$ equal to the identity outside a neighbourhood of each critical point (see Ref.~\cite{palais1959natural}).
%
\end{proof}

\section{Local Solutions} \label{sec:local}

Our goal is to count the number of harmonic $d_f+d^\Lambda_f$ forms. Given Proposition~\ref{prop:hd} we need only consider the primitive harmonic $k$-forms (which must satisfy $k \leq n$). Akin to Witten's original paper~\cite{witten1982supersymmetry}, the idea is that harmonic forms become localized to critical points of $f$ as $f$ becomes large. In this section we give a local description of the harmonic forms in this limit. Because we choose our critical points to be compatible, local solutions of the Witten Laplacian (quantum harmonic oscillator groundstates) are primitive forms, so are necessarily in the kernel of $\mathcal{D}_{Pf}$. The main effort of this section is in showing that no other local solutions exist. 

We choose the metric $g$ to be Euclidean on the neighbourhoods of $U_p$ of each critical point of $f$ and use the local coordinates of Lemma~\ref{lem:what}. Explicitly for $f$ we write $f = T \sum_{i} \lambda_i x_i^2/2$, where $\lambda_i = \pm 1$ and $T>0$ is real. To make the dependence on $T$ explicit, we will use the notation $\mathcal{D}_T$ for $\mathcal{D}_f$, which will become useful in Section~\ref{sec:global}. Additionally, we will write $\mathcal{D}_{PT}$ for the Witten-deformed version of $\mathcal{D}_P$.
\begin{prop} \label{prop:BIGP}
On $\mathbb{R}^{2n}$ with the a Euclidean metric, standard symplectic form and function $f$ with critical point of Morse index $n_p$ as given in Lemma~\ref{lem:what}, $L^2$ primitive solutions of $\mathcal{D}_{T} \alpha = 0$ are in correspondence with solutions of $\Delta_{d_f} \alpha = 0$ for $n_p \leq n$ and do not exist for $n_p >n$.  Explicitly, for $n_p\leq n$ the kernel of $\mathcal{D}_{T}$ (acting on primitive $k$-forms) is one-dimensional, generated by
$$\alpha = {\rm exp} \left ( -\frac{T}{2} \sum_{i=1}^{2n} x_i^2 \right ) dx_{1} \wedge dx_{3} \wedge \ldots \wedge dx_{2n_p-1}.$$
\end{prop}
To prove we proceed in stages. We first introduce the operators (in coordinates)
$D_i = (\partial_i + \partial_i f)$ and $D^\dagger_i = (-\partial_i + \partial_i f)$ with $\partial_i = \partial/\partial x_i$ differential operators in the local coordinates. These operators satisfy the commutation relations
 \begin{equation} \label{eq:comm} [D_i, D_j] = 0, \quad [D^\dagger_i, D^\dagger_j]=0, \quad [D_i, D^\dagger_j] = 2T \lambda_i \delta_{ij},\end{equation}
We also define the operators $e_i = dx_i \wedge$ and $e^\dagger_i = \iota_{X_i}$, where $X_i$ is the coordinate vector field associated to $x_i$. Then in terms of these operators, the differentials $d_f$, $d_f^\ast$, $d_f^\Lambda$ and $d_f^{\Lambda \ast}$ become
 $$ d_f = \sum_{i=1}^{2n} D_i a_i, \quad d^\ast =  \sum_{i=1}^{2n}  D^\dagger_i a^\dagger_i, \quad d^\Lambda = \sum_{i=1}^{2n} D_{Ji} a_{i}^\dagger, \quad d^{\Lambda \ast} =  \sum_{i=1}^{2n} D^\dagger_{Ji} a_{i},$$
where $D_{J(2i-1)} = D_{2i}$,  $D_{J(2i)} = -D_{2i-1}$ and similarly for $D^\dagger$ (note the minus sign). Now, as is well-known~\cite{witten1982supersymmetry} in such coordinates the Witten-deformed Hodge Laplacian takes the form
\begin{equation} \label{eq:witlap}
\Delta_{d_f} = d_f^\ast d_f + d_f d_f^\ast= \sum_{i=1}^{2n} \Big (- \partial_i^2 + T^2 x_i^2  - T \Big ) +T \sum_{i=1}^{2n} (1-\lambda_i + 2\lambda_i  e_i e^\dagger_i).
\end{equation}
We now give a local expression for the Laplacian associated to the $d^\Lambda$ cohomology.
\begin{lem} \label{lem:tt}
The local form of Laplacian associated to $d^\Lambda$ is written as
$$\Delta_{d_f^\Lambda} = d_f^{\Lambda \ast} d_f^\Lambda + d_f^\Lambda d_f^{\Lambda \ast} =  \sum_{i=1}^{2n} \Big (- \partial_i^2 + T^2 x_i^2  - T \Big ) +T \sum_{i=1}^{2n} (1+\lambda_{Ji} - 2\lambda_{Ji}  e_i e^\dagger_i),$$
where $\lambda_{J(2i-1)} = \lambda_{2i}$, and $\lambda_{J2i} = \lambda_{2i-1}$ (note the lack of minus sign). 
\end{lem}
\begin{proof}
Expressed in terms of the $D$ operators we find
 $$ \Delta_{d_f^\Lambda} = \sum_{i,j=1}^{2n} D_{Ji} D_{Jj}^\dagger a_i^\dagger a_j + D_{Jj}^\dagger D_{Ji}e_j e_i^\dagger $$
 We can then expand $\Delta_{d^\Lambda_f}$ as
 $$ \Delta_{d_f^\Lambda} = \sum_{i=1}^{2n} D_i D^\dagger_i - 2 T \sum_{i=1}^{2n} \lambda_{Ji} e_i e_i^\dagger.$$
Expanding the $D$ operators we find
  $$ \Delta_{d_f^\Lambda} = \sum_{i=1}^{2n} \Big(- \partial_i^2 + T^2(x_i)^2 -T \Big)  + T\sum_{i=1}^{2n} \Big ( 1+ \lambda_{Ji}  - 2 \lambda_{Ji} e_i e_i^\dagger \Big ).$$
\end{proof}
In the symbol calculation for $\mathcal{D}$ (Proposition~\ref{prop:D}) we made use of the K\"{a}hler relations $d d^{\Lambda \ast} \cong - d^{\Lambda \ast} d$, $d^\ast d^\Lambda \cong - d^\Lambda d^\ast$. In our local coordinates, these relations are promoted to equalities. However, they do not hold for the deformed differentials. 
\begin{lem} \label{lem:swi}
In the local coordinate system we have
$$d_f d_f^{\Lambda \ast} + d_f^{\Lambda \ast} d_f = -2 T \sum_{i=1}^n(\lambda_{2i-1}+\lambda_{2i}) e_{2i-1} e_{2i} = C_f $$
$$d^\ast_f d_f^{\Lambda} + d_f^{\Lambda } d^\ast_f = 2T  \sum_{i=1}^n(\lambda_{2i-1}+\lambda_{2i}) e^\dagger_{2i-1} e^\dagger_{2i} = C^\dagger_f $$
\end{lem}
\begin{proof}
For the first relation we have
$$ d_f d_f^{\Lambda \ast} + d_f^{\Lambda \ast} d_f = \sum_{i,j=1}^{2n} [D_{i}, D^\dagger_{Jj}] e_i e_j,$$
separating out into sums up to $n$ we find
$$ d_f d_f^{\Lambda \ast} + d_f^{\Lambda \ast} d_f= \sum_{i,j=1}^n \Big ( [D_{2i}, D^\dagger_{J(2j-1)}] e_{2i} e_{2j-1}+ [D_{2i-1}, D^\dagger_{J(2j)}] e_{2i-1} e_{2j}\Big ).$$
Which then becomes
$$ d_f d_f^{\Lambda \ast} + d_f^{\Lambda \ast} d_f= \sum_{i,j=1}^n \Big ( [D_{2i}, D^\dagger_{2j}] e_{2i} e_{2j-1}- [D_{2i-1}, D^\dagger_{2j-1}] e_{2i-1} e_{2j}\Big ),$$
using the commutation relations \eqref{eq:comm} we then find
$$ d_f d_f^{\Lambda \ast} + d_f^{\Lambda \ast} d_f = -2 T \sum_{i=1}^n (\lambda_{2i-1}+\lambda_{2i})e_{2i-1}e_{2i}.$$
For the second relation we have
$$ d^\ast_f d_f^{\Lambda } + d_f^{\Lambda} d^\ast_f = \sum_{i,j=1}^{2n} [D^\dagger_{i}, D_{Jj}] e^\dagger_i e^\dagger_j,$$
and a similar calculation yields the result.
\end{proof}
We now give some relations for the $C_f$ and $C^\dagger_f$ operators. We define the operator
$$M_f = \Delta_{d_f}-\Delta_{d^\Lambda_f}=   2 T \sum_{i=1}^n(\lambda_{2i-1} + \lambda_{2i}) ( e_{2i-1} e^\dagger_{2i-1}+ e_{2i} e^\dagger_{2i}-1).$$
The following identities are easily shown.
\begin{lem} \label{lem:HUB}
The operators $C_f$, $C_f^\dagger$, and $M_f$ satisfy
\begin{align*}
&[d_f^\Lambda, M_f] = - [d_f, C_f^\dagger], &\quad  [d_f, M_f] &= [d_f^\Lambda, C_f], \quad & [M_f, \Lambda] & = 2 C_f^\dagger, \\
&[d_f^\ast, M_f] =  [d_f^{\Lambda \ast}, C^\dagger_f], &\quad  [d_f^{\Lambda \ast}, M_f] &= - [d_f^\ast, C_f],  \quad & [M_f, L] &= -2C_f .
\end{align*}
\end{lem}
We are now in a position to give an expression for $\mathcal{D}_{PT}$
\begin{lem} \label{lem:op}
The operator $\mathcal{D}_{PT}$ has the local form 
$$\Delta_{d_f}^2  + d_f^{\Lambda \ast} d_f  C_f^\dagger + C_f  d_f^\ast d^\Lambda_f- C_f C^\dagger_f - d_f M_f d^\ast_f.$$
\end{lem}
\begin{proof}
The operator $\mathcal{D}_{PT}$ is given as
$$ \mathcal{D}_{PT} = d_f^\ast d_f d_f^\ast d_f + d_f^{\Lambda \ast} d_f d_f^\ast d_f^\Lambda+  d_f d_f^\Lambda d_f^{\Lambda \ast} d_f^\ast.$$
From Lemma~\ref{lem:tt} we have
\begin{equation} \label{eq:what} \Delta_{d_f} = \Delta_{d^\Lambda_f} +M_f.
\end{equation}
Applying $d_f d^\ast_f$ to \eqref{eq:what} on the right and using the relations in Lemma~\ref{lem:swi} we find
\begin{equation} \label{eq:reap1a} 
d^{\Lambda \ast}_f d_f d_f^\ast d_f^\Lambda + d_f d_f^\Lambda d_f^{\Lambda \ast} d_f^\ast  = d_f d_f^\ast d_f d_f^\ast - M_f d_f d^\ast_f  + d_f^{\Lambda \ast} d_f  C_f^\dagger -  d^\Lambda_f C_f   d_f^\ast.
\end{equation}
which allows $\mathcal{D}_{PT}$ to be written as
$$ \mathcal{D}_{PT} = \Delta_{d_f}^2 - M_f d_f d^\ast_f  + d_f^{\Lambda \ast} d_f  C_f^\dagger -  d^\Lambda_f C_f   d_f^\ast.$$
Using the commutation relation $[d_f, M_f] = [d_f^\Lambda, C_f]$ we find 
$$ \mathcal{D}_{PT} = \Delta_{d_f}^2 - M_f d_f d^\ast_f  + d_f^{\Lambda \ast} d_f  C_f^\dagger -  C_f d^\Lambda_f   d_f^\ast - [d_f, M_f] d_f^\ast,$$
finally applying the relation  $d^\Lambda_f   d_f^\ast +  d_f^\ast d^\Lambda_f  = C_f^\dagger$ yields the result.
\end{proof}
\begin{lem} \label{lem:sat}
Any local solution must satisfy
$$ \| \Delta_{d_f} \alpha \|^2 - \|C^\dagger \alpha \|^2 - (d_f^\ast \alpha, M_f  d_f^\ast \alpha)_g  =0.$$
\end{lem}
\begin{proof}
First we write
$$ (\alpha, \mathcal{D}_{PT} \alpha) = \| \Delta_{d_f} \alpha \|^2 - \|C^\dagger \alpha \|^2 - (d_f^\ast \alpha, M_f  d_f^\ast \alpha) + (d^\Lambda_f \alpha , d_f C_f^\dagger \alpha).$$
Any local solution must satisfy $(\alpha \mathcal{D}_{Pf} \alpha) =0$ and $d^\Lambda_f \alpha=0$, so this becomes
$$ (\alpha, \mathcal{D}_{PT} \alpha) = \| \Delta_{d_f} \alpha \|^2 - \|C^\dagger \alpha \|^2 - (d_f^\ast \alpha, M_f  d_f^\ast \alpha) =0.$$
\end{proof}
\begin{proof}[Proof of Proposition~\ref{prop:BIGP}]
Let $\alpha$ be a primitive $k$-form (which forces $k \leq n$). A general element of $L^2(\Omega^k(\mathbb{R}^{2n}))$, denoted $\alpha$, may be written as
\begin{equation} \label{eq:whooo}
\alpha =\sum_{I,J}c_{IJ} \Psi_I dX_J,
\end{equation}
where $\Psi_I$, $I \in \mathbb{N}_0^{2n}$ is a normalised eigenfunction of the $2n$-dimensional quantum harmonic oscillator, $c_{IJ} \in \mathbb{R}$, and each $J$ is a $k$-tuple, so that
$$ dX_J = \bigwedge_{j \in J} dx_j.$$ 
As we will show, it suffices to consider $\alpha$ to be of the form
\begin{equation} \label{eq:whooob}
\alpha = \psi_I \sum_{J} c_J dX_J,
\end{equation}
where $c_J \in \mathbb{R}$ and we choose an overall normalisation,
$$ \sum_{J} c_J^2 = 1.$$
From Lemma~\ref{lem:sat} a necessary condition for $\alpha$ to be in the kernel of $\mathcal{D}_{Pf}$ is to solve
\begin{equation} \label{eq:zero}\| \Delta_{d_f} \alpha \|^2 - \| C^\dagger \alpha \|^2 - (d_f^\ast \alpha, M_f  d_f^\ast \alpha)  =0.\end{equation}
We will separately treat the cases $n_p<n$, $n_p=n$, $n_p>n$.
\subsection*{The case $\bm{ n_p=n}$} This is immediate as if $n_p=n$, $M_f=C_f=C^\dagger_f=0$, so we must have $\Delta_{d_f} \alpha=0$. 
\subsection*{The case $\bm{n_p <n}$}
Here we have
$$\lambda_{2i-1}+\lambda_{2i}= \begin{cases} 0 & \quad 0\leq i \leq n_p, \\
2 & \quad n_p < i \leq n. \end{cases}$$
Now given a $k$-form $dX_J$ we define two numbers. $R^+(J)$ counts the number of $dx_j$ in $j$ with $j > 2n_p$, $R^0(J)$ counts the number of $dx_j$ equal to $dx_{2i}$ where $i\leq n_p$. Moreover, let $\eta_I = \sum_{i \in I} I_i$. 
We now consider the three terms in \eqref{eq:zero} applied to $\alpha$ separately. Considering the explicit form of the $M_f$ operator allows one to write
$$(\alpha, d_f M_f \delta_f \alpha) = 4t(n_p-n) \| \delta_f \alpha \|^2 + 4T \sum_{i=2n_p+1}^{2n} \| e_i^\dagger \delta_f \alpha \|^2$$
Now assuming the form of $\alpha$ given we use the triangle inequality to write
$$(\alpha, d_f M_f \delta_f \alpha) \leq 4t(n_p-n) \| \delta_f \alpha \|^2 + 4T \sum_J c_J^2 \sum_{i=2n_p+1}^{2n} \| e_i^\dagger \delta_f \Psi_I dX_J  \|^2.$$
(Observe that if we had multiple $I$ indices, these would also separate out in this way, and we are justified in having a single $I$ term in \eqref{eq:whooob}.) Now $ \| e_i^\dagger \delta_f \Psi_I dX_J  \|^2=0$ if $e_i^\dagger dX_J=0$ and otherwise is bounded from above by $\| \delta_f \Psi_I dX_J  \|^2$. Hence we may write
$$(\alpha, d_f M_f \delta_f \alpha)g \leq 4T(n_p-n) \| \delta_f \alpha \|^2 + 4T \sum_J c_J^2 R^+(J) \| \delta_f \Psi_I dX_J  \|^2$$
Now, since any solution must have $d_f \alpha=0$, we must also have, assuming $\alpha$ is a solution
\begin{align*}
&(\alpha, d_f M_f \delta_f \alpha) \leq \\ & 4T(n_p-n) (\| \delta_f \alpha \|^2+\|d_f \alpha \|^2) + 4T \sum_J c_J^2 R^+(J)( \| \delta_f \Psi_I dX_J  \|^2+\| d_f \Psi_I dX_J  \|^2)
\end{align*}
which then becomes
$$(\alpha, d_f M_f \delta_f \alpha) \leq 4T \sum_{J}c_J^2(n_p-n + R^+(J)) W_{IJ},$$
where
$$W_{IJ} = 2T \left ( \eta_I+  (n_p-k+ 2(R^0(J) + R^+(J))) \right ),$$
is the eigenvalue of the Witten Laplacian operating on $\Psi_I dX_J$. The first term in \eqref{eq:zero} is readily shown to be
$$  \sum_{J}c^2_J W_{IJ}^2. $$
For the second term we give an upper bound on the value of $\| C^\dagger \alpha \|^2 $ from Lemma~\ref{lem:bigo} with $a=0$ which is 
$$ \| C^\dagger_f \alpha \|^2 \leq 8 T^2 \sum_J c_J^2 \left ( 2 \lfloor{ R^+(J) /2 \rfloor}  \right ).$$
This then yields (assuming $d_f \alpha = 0$)
\begin{align*} & \| \Delta_{d_f} \alpha \|^2 - \| C^\dagger \alpha \|^2 - (d_f^\ast \alpha, M_f  d_f^\ast \alpha)_g \geq  \\ & T^2 \sum_J c_J^2 \left ( W_{IJ}(W_{IJ} - n_p-n+R^+(J) )- 16 \lfloor R^+(J)/2 \rfloor \right ) 
\end{align*}
Our goal is to show that this expression can only vanish if $R^+(J) = R^0(J)=0$, since this implies that the any local solution is in the kernel of the Witten Laplacian. Now we may assume, without loss of generality, that $\eta_I=0$, since $\eta_I$ only increases the expression. Finally, noting that $W_{IJ} \geq 2R^+(J)$, we arrive at the form
\begin{align*} &\| \Delta_{d_f} \alpha \|^2 - \| C^\dagger \alpha \|^2 - (d_f^\ast \alpha, M_f  d_f^\ast \alpha)_g \geq  \\  & T^2 \sum_J c_J^2 \left ( 2 R^+(J)  (4 n - 2 (n_p+k)+ 4R^0(J)) - 16 \lfloor R^+(J)/2 \rfloor) \right)
\end{align*}
Now observe that $ (4 n - 2 (n_p+k)+ 4(R^0(J))) > 2n$, hence if $n \geq 2$, the above expression is greater than zero unless $R^+(J)=0$ for all $J$. If $n=1$, then $R^+(J) \leq 1$, which implies that the above expression is also positive unless $R^+(J)=0$. Now if $R^+(J)=0$, we first observe that $C^\dagger \alpha = 0$. Moreover, in this case $M_f$ simply multiplies by $-4T(n-n_p)$, hence the only solutions to \eqref{eq:zero} have $\Delta_{d_f} \alpha = 0$, which also requires $d^\ast_f \alpha =0$.

\subsection*{The case $\bm{n_p > n}$}
This time we have
$$\lambda_{2i-1}+\lambda_{2i}= \begin{cases} 0 & \quad 0\leq i \leq 2n-n_p, \\
-2 & \quad 2n-n_p < i \leq n. \end{cases}$$
Again, given a $k$-form $dX_J$ we (re)define several associated numbers. $R^-(J)$ counts the number of $dx_j$ with $j>4n-2n_p$, $R^0(J)$ counts the number of $dx_{j}$ equal to $dx_{2i-1}$ with $i \leq 2n-n_p$, and $R^+(J)$ counts the numbers of $dx_{j}$ equal to  to $dx_{2i}$ with $i \leq 2n-n_p$. We once again consider the three terms in \eqref{eq:zero} separately. This time the third term is written as
$$(\alpha, d_f M_f \delta_f \alpha) = 4T(n_p-n) \| \delta_f \alpha \|^2 - 4T \sum_{i=4n-2n_p+1}^{2n} \| e_i^\dagger \delta_f \alpha \|^2,$$
which we rewrite as
\begin{equation} \label{eq:haha}(\alpha, d_f M_f \delta_f \alpha) = 4T(n_p-n) \| \delta_f \alpha \|^2 - 4T \sum_{i=4n-2n_p+1}^{2n} \left (\left ( \|  \delta_f e_i^\dagger \alpha \|^2 + \|  d_f e_i^\dagger \alpha \|^2 \right) - \| d_f e_i^\dagger \alpha  \|^2\right)
\end{equation}
The first term in \eqref{eq:haha} can be bounded from below as
$$ \| \delta_f \alpha \|^2 \leq \sum_J c_J^2 \| \delta_f \psi_I dX_J \|^2 \leq 2T \sum_{J} c_J^2 (\eta_I +  R^+(J))$$
Now consider the second term in \eqref{eq:haha}
$$ \|  \delta_f e_i^\dagger \alpha \|^2 + \|  d_f e_i^\dagger \alpha \|^2  = ( \alpha, e_i \Delta_{d_f} e_i^\dagger \alpha) = \sum_J c_J^2 {\tilde W}_{i,IJ}$$
where
$$ {\tilde W}_{i,IJ} = \begin{cases} W_{IJ}+2T & \quad i \in J \\ 
0 & \quad i \notin J \end{cases}. $$
This allows \eqref{eq:haha} to be written as
\begin{align*}
& (\alpha, d_f M_f \delta_f \alpha)g \leq \\ &  4T \sum_J \left( c_J^2 (2T (n_p-n)(\eta_I+R^+(J)) - R^-(J)(W_{IJ}+2T)  \right) + 4T \sum_{i=4n-2n_p+1}^{2n}  \| d_f e_i^\dagger \alpha \|^2 
\end{align*}
The final piece here can be written, assuming $d_f \alpha=0$, as
$$ 4T \sum_{i=4n-2n_p+1}^{2n}  \| d_f e_i^\dagger \alpha \|^2 =  4T \sum_{i=4n-2n_p+1}^{2n}  \| D_i \alpha \|^2    \leq \sum_J c_J^2 \sum_{i=4n-2n_p+1}^{2n} \| D_i \Psi_I dX_J \|.$$
Now observe that $D_i$ acts as an annihilation operator if $x_i$ is a positive direction, and a creation operator if $x_i$ is a negative direction. Hence we have
$$ \sum_{i=4n-2n_p+1}^{2n} \| D_i \Psi_I dX_J \| = 2T(\eta_I + 2(n_p-n)).$$
Putting it together we find the third term is bound from below as
\begin{align*} &(\alpha, d_f M_f \delta_f \alpha) \leq \\ & 4T \sum_J \left( c_J^2 (2T (n_p-n)(\eta_I+R^+(J)) - 4T {\rm Max}\left [ 0,R^-(J)(W_{IJ}+2T)  -2T(\eta_I + 2(n_p-n))\right ] \right) \end{align*}
The first term in \eqref{eq:zero} is now written as
$$ \sum_{J} c_J W_{IJ}^2$$ with
$$ W_{IJ} = 2T(\eta_I + n_p-k+2 R^+(J) )$$
For the third term in \eqref{eq:zero} we may simply replace $R^+$ with $R^-$ and (again taking $a=0$) we find
$$ \| C^\dagger_f \alpha \|^2 \leq 8 T^2 \sum_J c_J^2 \left ( 2 \lfloor{ R^-(J) /2 \rfloor}  \right ).$$
Putting it all together we obtain the inequality
\begin{align*}
(\alpha, D_{Pf} \alpha) \geq & \\ &  4T^2 \sum_J  c_J^2 \Big( \frac{1}{2}(n-n_p)(k+\rho_J)+(n_p+\rho_J)^2 - 4 \lfloor{ R^-(J) /2 \rfloor}+ \\ & 2 {\rm Max}\left [ 0, 2n+n_p(R^-(J)-2)+R^-(J)(1+\rho_J) \right]\Big),
\end{align*}
where $\rho_J = R^+(J)-R^0(J)-R^-(J)$. (We may once again assume $\eta_I=0$, since it only increases the expression). The resulting expression is minimised if $\rho_J=-k$, and is then easily shown to be strictly positive (with a minimum, for each $J$, of 1 if $n_p-n=1$, $R^-(J)=1$). Since there is a positive lower bound, there can be no solutions.
\end{proof}
\begin{lem} \label{lem:bigo}
Let $\alpha = \Psi_I \sum_J c_J dX_J$ be a primitive $k$-form, as in \eqref{eq:whooob}. Let $n_p < n$. Then for all $a \in \mathbb{R}$,
$$ \| C^\dagger_f \alpha \|^2 \leq 8 T^2 \sum_J c_J^2 \left ( 2 \lfloor{ R^+(J) /2 \rfloor}  (1-a)^2+(k-R^+(J))a^2 \right ).$$
\end{lem}
\begin{proof}
Consider the space of primitive $k$-forms, $P\Omega^k(\mathbb{R}^{2n})$ (recall we have the standard symplectic structure and Euclidean metric). Then there is a Lefschetz-type decomposition
$$P\Omega^k(\mathbb{R}^2n) = \bigoplus_r PC\Omega^{k-2r}(\mathbb{R}^{2n}) \wedge Z^r(\mathbb{R}^{2n}),$$
which is orthogonal with respect to the inner product. The space $PC\Omega^{k}(\mathbb{R}^{2n})$ corresponds to `coisotropic' $k$-forms, and has dimension $2^n {{n} \choose{k}}$. A basis is given by the $dX_J$ such that, for all $i,j \in J$,
$$ \left \lfloor \frac{i}{2} \right  \rfloor \neq  \left \lfloor \frac{j}{2} \right \rfloor.$$
Let $z_i = dx_{2i-1} \wedge dx_{2i}$, the $Z^k(\mathbb{R}^{2n})$ are primitive $2k$-forms generate by the $z_i$. To describe a basis for these, let $S$ be an $n$-simplex, we associate $z_i$ to the $i^{th}$ vertex of $S$, $z_i \wedge z_j$ to the edge between nodes $i$ and $j$, and so on. Then an element of of $Z^k(\mathbb{R}^{2n})$ can be identified with a $k$-chain in $S$. The dual Lefschetz operator $\Lambda$ then appears as an unsigned boundary operator on the simplex $S$, for example
$$\Lambda (z_i \wedge z_j ) = z_i + z_j,$$
with similar formulae for higher order forms. A basis of $Z^k(\mathbb{R}^{2n})$ can be identified with the set of $k$-chains which are closed under the unsigned boundary operator $\Lambda$, (this is just the standard boundary operator in simplicial homology, without the factor of $(-1)^i$). The dimension of this space is therefore
$$ {\rm dim} \, Z^k(\mathbb{R}^{2n}) =  {{n}\choose{k}} - {{n}\choose{k-1}}.$$
Given a primitive form written as $c \wedge \zeta$, with $c$ coisotropic, and $\zeta$ an element of some $Z^k$, observe that the $C_f^\dagger$ operator respects this decomposition,
$$ C_f^\dagger (c \wedge \zeta)= c \wedge C_f^\dagger \zeta$$
Hence we need only consider the operation of $C^\dagger_f$ on $Z^k(M)$. We describe it in terms of the simplicial picture. The map $C^\dagger_f$ is ($4T$ times) an unsigned boundary operator which only counts those $z_i$ with $i>n_p$.

We now consider $\|C_f^\dagger \alpha \|^2$. Since any cosisotropic pieces of $\alpha$ are unaffected by $C_f^\dagger$, the norm squared decomposes as a sum over all terms in $\alpha$ with the same coisotropic piece, so we may assume
$$\alpha = c \wedge \zeta.$$
We say $\zeta$ has rank $\tilde k \leq k$. We then need to consider the norm squared of $C_f^\dagger \zeta$. This will be a sum over each $(\tilde k-1)$-cell of $S$ of the image of $C_f^\dagger$ sitting there. Since $\alpha$ is primitive, we can instead use the operator $C_f^\dagger - 4 T a \Lambda$, which has the same image. We then use the triangle inequality on the norm at each $(\tilde k-1)$ cell. Observe that the total contribution from a given $z_j$ will be $16T^2(1-a)^2$ if $j>n_p$ and $16 T^2 a^2$ if $j \leq n_p$. This then gives
$$ \| C_f^\dagger \alpha \|^2 \leq 16 T^2 \sum_J c_J^2( \tilde R^+(J) (1-a^2) + (\tilde k - \tilde R^+(J))a^2),$$
where $\tilde R^+(J)$ counts twice the number of $z_i$, $i>n_p$, appearing in $dX_J$. Now observe that $\tilde R^+(J) \leq  \lfloor{ R^+(J) /2 \rfloor} $ and $(\tilde k - \tilde R^+(J)) \leq ( k -  R^+(J))$. The result then follows.
\end{proof}
\newpage
\section{Global Solutions} \label{sec:global}
In this section we fix some $k$, and let ${\bm H}^l(M)$ be the Sobolev space $W^{2,l}(P\Omega^k(M))$ of primitive $k$-forms, with $\| \cdot \|$ the $l=0$ norm, and $(\cdot, \cdot)$ the inner product. Moreover, in this section consider the Bilaplacian $\mathcal{D}_T$ as a differential operator
$$\mathcal{D}_T : {\bm H}^4(M) \to {\bm H}^0(M).$$
The purpose of this section is to use the local solutions from Section~\ref{sec:local}, along with some analytic estimates, to prove the following.
\begin{prop} \label{prop:vals}
For every $c>0$ there exists a $T_c$ such that for $T>T_c$ the number of eigenvalues of $\mathcal{D}_{T}$ in the range $[0,c]$ is equal to $m_k$.
\end{prop}
The proof follows that given by Zhang~\cite{zhang2001lectures} (Chapters 4 and 5). Let $Z_{k}(f)$ be the set of critical points of $f$ on $M$ with Morse index $k$. As in Section~\ref{sec:local} we choose the metric $g$ on $M$ to be Euclidean in a neighbourhood of each critical point $p \in Z(f)$. Let $U_p$ be an open neighbourhood of $p$ such that $g$ is Euclidean on $U_p$. Without loss of generality we take $U_p$ to be an open ball around $p$ with radius $6a$, and we will denote by $U_p(r)$ the open ball of radius $r$ (also Euclidean). Now define a smooth cutoff function $\gamma: \mathbb{R} \to [0,1]$ such that
$$ \gamma(x) = \begin{cases} 1 & \quad |x| \leq a, \\ 0 & \quad |x| \geq 2a. \end{cases}$$
We then extend $\gamma$ to a function on $M$ by setting it to zero outside each $U_p(2a)$. Now for any critical point $p \in Z(f)$, $T>0$, define 
$$ \sigma_{p,T} = \frac{\gamma(|x|)}{\sqrt{ \alpha_{p,T}}} \tilde \sigma_{p,T}$$
where $\tilde \sigma_{p,T}$ is equal to the local solution described in Proposition~\ref{prop:BIGP} for $x \in U_p$ and 0 otherwise. $\sigma_{p,T}$ is then a primitive $k$-form. The constant $\alpha_{P,T}$ is chosen such that
$$\|  \sigma_{p,T} \|=1.$$
Now, let $E_{T,k}$ be the direct sum of vector spaces generated by the $\sigma_{p,T}$s. Moreover, let $E^\perp_{T}$ be the orthogonal complement of to $E_T$ in ${\bm H}^0(M)$. Then there is an orthogonal splitting
$${\bm H}^0(M) = E_{T} \oplus E^\perp_{T}.$$
Additionally, let $\pi_T$ and $\pi_T^\perp$ denote the orthogonal projections from ${\bm H}^0(M)$ onto $E_T$ and $E_T^\perp$ respectively. We then define a splitting of the operator $\mathcal{D}_f$ (see Bismut and Lebeau~\cite{bismut1991complex} Chapter IX) as
$$ \mathcal{D}_f = \mathcal{D}_{T,1}+ \mathcal{D}_{T,2}+\mathcal{D}_{T,3}+\mathcal{D}_{T,4},$$
where
\begin{align*}
\mathcal{D}_{T,1} = \pi_T D_T \pi_T, &\quad \mathcal{D}_{T,2} = \pi_T D_T \pi^\perp_T, \\
\mathcal{D}_{T,3} = \pi^\perp_T D_T \pi_T, &\quad \mathcal{D}_{T,4} = \pi^\perp_T D_T \pi^\perp_T.
\end{align*}
\newpage
We then have the following estimates.
\begin{prop} \label{prop:est} \hspace{2ex}\newline
\begin{enumerate}[label=(\roman*)]
\item There is a constant $T_0>0$ such that for $T>T_0$ and any $s \in {\bm H}^4(M)$ we have
\begin{align*}(s, \ \mathcal{D}_{T,1}  s)_0 &\leq \frac{1}{T^2} \| s\|_0^2,\\
(s, \ \mathcal{D}_{T,2}  s)_0= (s, \ \mathcal{D}_{T,3}  s)_0 &\leq \frac{1}{T^2} \| s\|_0^2.
\end{align*}
\item There exists $T_1>0$ and $C>0$ such that for any $s \in E_T^\perp \cap {\bm H}^4(M)$ and $T \geq T_1$,
$$ (s, \mathcal{D}_{f} s ) \geq C T \|s\|^2.$$
\end{enumerate}
\end{prop}
\begin{proof}
For $(i)$ we have
$$\pi_T s  = \sum_{p\in Z_k(f)} (s, \sigma_{p,T}) \sigma_{p,T},$$
and then
For each $p$, $\mathcal{D}_f \sigma_{p,T}$ will be concentrated in an annulus of radius $a \leq r \leq 2a$ around $p$. It follows that
$$(s, \ \mathcal{D}_{T,1}  s) = (s, \pi_T \mathcal{D}_f \pi_T s) = \sum_{p \in Z_k(f)} (s, \sigma_{p,T})^2 (\sigma_{p,T}, \mathcal{D}_f \sigma_{p,T})  $$
Now, since $\sigma_{p,T}$ has support on $U_p(2a)$, we may use the Euclidean metric, so that the eigenvalues of $\mathcal{D}_f$ are proportional to $T^2$. Therefore, there is some constant $C$ such that 
$$ \langle \sigma_{p,T}, \mathcal{D}_f \sigma_{p,T} \rangle \leq C T^2 e^{-r T},$$
where $r$ is the radial coordinate in Euclidean space. Hence there is some $C_1>0$ such that
$$(\sigma_{p,T}, \mathcal{D}_f \sigma_{p,T} ) \leq C_1 T^{2-n} e^{-T}.$$
Hence, for $T$ sufficiently large
$$(\sigma_{p,T}, \mathcal{D}_f \sigma_{p,T} ) \leq \frac{1}{T^2},$$
which gives the first inequality in $(i)$ as
$$ (s, \ \mathcal{D}_{T,1}  s) \leq \frac{1}{T^2} \| s\|^2.$$
For the second inequality, note that $\mathcal{D}_{T,2}$ and $\mathcal{D}_{T,3}$ are formal adjoints of each other, so are equal. Then we have
$$ (s, \mathcal{D}_{T,3} s) = \sum_{p \in Z_{k}(f)} (\pi^\perp_T s, \mathcal{D}_T \sigma_{p,T})(s, \sigma_{p,T})$$
Now observe that, once again, that the eigenvalues of $\mathcal{D}_T$ are proportional to $T^2$. Moreover, the support of $(\pi^\perp_T s, \mathcal{D}_T \sigma_{p,T})$ is an annulus around each critical point. The result follows by analogous reasoning to the previous case.

We now prove $(ii)$. We first define a new function $\gamma_2$ on $M$, equal to $\gamma(r/2)$ for $r \leq 4a$, with $r$ a radial coordinate about each $p \in Z(f)$ (note this includes all critical points, not just index $k$) and set $\gamma_2=0$ otherwise. We may then write a general $s \in E_T^\perp \cap {\bm H}^4(M)$ as
$$ s = \gamma_2 s + (1-\gamma_2)s = s_1+s_2.$$
Both $s_1$ and $s_2$ are elements of $ E_T^\perp \cap {\bm H}^4(M)$. Now observe that the support of $s_1$ is a collection of balls of radius $4a$ about each critical point of $f$, which are Euclidean. Hence we may write $s_1$ as a sum over local eigenvectors for all critical points. A simple scaling argument shows that these are proportional to $T^2$ so we may write 
$$s_1 = \sum_a c_a \psi_a, \quad s_1 =T^2 \sum_a \lambda_a c_a \psi_a, $$
Now we can extend each $\psi_a$ to a function on $M$ by multiplying by a smooth bump function going from one to zero between radii of $4a$ and $6a$ around each critical point. Then we have
$$(s, \mathcal{D}_T s) = (s_2, \mathcal{D}_T s_2) + T^2 \sum_a \lambda_a(c_a \psi_a + 2 s_2, c_a \psi_a).$$
Now all the $T=1$ eigenvalues $\lambda_a$ are positive except for the single zero eigenvalue at each degree $k$ critical point. Let $C_4$ be the smallest non-zero local eigenvalue. We may then write
$$(s, \mathcal{D}_T s) \geq (s_2, \mathcal{D}_T s_2) + C_4 T^2 \sum_a (c_a \psi_a + 2 s_2, c_a \psi_a) - C_4 T^2 \sum_{p \in Z_k(f)}(c_p+\psi_p+2s_2, c_p \psi_p).$$
Now the support of $s_2$ is on the complement of the critical points of $f$. A short calculation shows that
$$ \mathcal{D}_T = \mathcal{D}_0 + T \mathcal{A} + T^2 \mathcal{B} + T^3\mathcal{C}+ 2T^4 | d f |^4, $$
where $\mathcal{A}$, $\mathcal{B}$, $\mathcal{C}$ are differential operators. Now outside the neighbourhoods of each critical point there is some constant $V$ such that $| d f |^4 \geq V$ on $M \setminus \cup_{p \in Z(f)} U_p$. It follows that
$$ \frac{(s_2, \mathcal{D}_T s_2)}{T^4} \geq V \| s_2 \|^2+ O(1/T)$$
and hence that for arbitrary $C_3$ there is a $T_3$ such that for $T \geq T_3$
$$ (s_2, \mathcal{D}_T s_2) \geq C_3 T^2 \| s_2 \|^2$$
It follows then, that there is a $T_4$ such that for $T\geq T_4$
\begin{align*}
(s, \mathcal{D}_T s)&\geq T^2 \sum_a C_4 (s_2 + c_a \psi_a,s_2+ c_a \psi_a) - T^2C_4 \sum_{p \in Z_{k}(f)}(c_p \psi_{0,p}, 2s_2+c_p \psi_{0,p}) \\
 &= T^2 C_4 \|s \|^2 - T^2 C_4 \sum_{p \in Z_{k}(f)}(c_p \psi_{0,p}, 2s_2+c_p \psi_{0,p}) 
\end{align*}
Now we have
$$c_p  =(s_1, \psi_{0,p}),$$
but since $(s_1, \sigma_{p,T})=0$ by assumption, we must have 
$$c_p = (s_1, (1-\gamma) \psi_{0,p})$$
Since $\gamma =1$ near $p$ and $\psi_{0,p}$ decays exponentially, we see that there is a $C_1$ and $T_1$ such that when $T\geq T_1$
\begin{equation} \label{eq:BLOOOOO}
 c_p^2 \leq \frac{C_1}{T} \| s_1 \|^2.
 \end{equation}
Moreover, since $s_2$ is zero on $U_p(2a)$, there must be a $C_2$ such that when $T>T_2$,
$$|(s_2, \psi_{0,p})| \leq \frac{C_2}{T} \|s_2 \|^2$$
Hence there is some $T_5$ such that, for $T\geq T_5$ we have
$$(s, \mathcal{D}_T s) \geq T^2 C_4 \|s \|^2 - T C_4( C_1 \|s_1\|^2+ 2 C_2 \|s_2 \|^2),$$
which implies there is a $C_5$ and $T_6$ such that for $T \geq T_6$,
$$(s, \mathcal{D}_T s) \geq T C_5 \|s \|^2. $$
\end{proof}
Now for any constant $c>0$, let $E_T(c)$ denote the direct sum of eigenspaces of $\mathcal{D}_{T}$ associated with eigenvalues lying in $[-c,c]$. $E_T(c)$ is a finite-dimensional subspace of ${\bm H}^0(M)$. Let $\pi_T(c)$ denote the orthogonal projection operator from ${\bm H}^0(M)$ to $E_T(c)$. 
\begin{lem} \label{lem:58}
There are constants $C>0$ and $T_1>0$ such that for any $T \geq T_1$ and $s \in E_T$,
$$ \| \pi_T(c)s - s \| \leq \frac{C}{T} \| s \|.$$
\end{lem}
\begin{proof}
Let $\delta = \{ \lambda \in \mathbb{C} \, : \, |\lambda | = c\}$ be the counter-clockwise oriented circle.  Let $s^\prime \in{\bm H}^4(M)$ be arbitrary, then we have  
$$ (s^\prime, (\mathcal{D}_T - \lambda) s^\prime)  = \sum_i (s^\prime, \mathcal{D}_{T,i} s^\prime) - \lambda \|s^\prime \|^2$$
Now, using the estimates in Proposition~\ref{prop:est} we see that when $T$ is sufficiently large the above expression is non-zero. Hence the resolvent
$$ (\mathcal{D}_T - \lambda)^{-1},$$
is well-defined for succifiently large $T$. Moreover, we see that for $T$ sufficiently large there is some constant $C_1$ such that
$$ \|  (\mathcal{D}_T - \lambda) s^\prime \| \geq C_1 \| s\|.$$
Then by the basic spectral theorem in operator theory~\cite{douglas2012banach}, we have
$$ \pi_T(c)s - s = \frac{1}{2 \pi i} \int_\delta ((\lambda - \mathcal{D}_T)^{-1} - \lambda^{-1}) s \,d \lambda . $$ 
Now we have (recall $s \in E_T$)
$$ \left ( ( \lambda - \mathcal{D}_T)^{-1} - \lambda^{-1} \right) s = \lambda^{-1} (\lambda - \mathcal{D}_T)^{-1}(\mathcal{D}_{T,1}+\mathcal{D}_{T,3}) s$$
We then observe that for $T$ sufficiently large
$$\|(\lambda - \mathcal{D}_T)^{-1}(\mathcal{D}_{T,1}+\mathcal{D}_{T,3}) s\| \leq \frac{1}{C_1} \| (\mathcal{D}_{T,1}+\mathcal{D}_{T,3}) s\|$$
Using the Cauchy-Schwartz inequality we then see there is a $C_2$ such that for $T$ sufficiently large
$$\|(\lambda - \mathcal{D}_T)^{-1}(\mathcal{D}_{T,1}+\mathcal{D}_{T,3}) s\| \leq \frac{1}{C_2 T} \|  s\|,$$
the result follows.
\end{proof}
\begin{rem}
As in Ref.\cite{zhang2001lectures}, in the above proof we complexify the spaces and extend the operators accordingly, however one can remain in the real category by taking the real parts above.
\end{rem}

\begin{proof}[Proof of Proposition~\ref{prop:vals}]
Lemma~\ref{lem:58} implies that for $T$ sufficiently large, the set of $\pi_T(c) \sigma_{p,T}$, $p \in Z_k(f)$ are linearly independent. Thus, for $T$ sufficiently large
$$ {\rm dim}\, E_T(c)  \geq {\rm dim} \, E_T.$$
Now assume $ {\rm dim}\, E_T(c)  > {\rm dim} \, E_T$. This implies that there is some $s \in E_T(c)$ which is orthogonal to $P_T(c)E_T$. That is
$$(s, P_t(c) \sigma_{p,T} )=0,$$
for all $p \in Z_k(f)$. This then implies
\begin{align*} \pi_T s &= \sum_{p \in Z_k(f)} (s, \sigma_{p,T}) \sigma_{p,T} - \sum_{p \in Z_{k}(f)} (s, \pi_T (c) \sigma_{p,T}) \pi_T(c) \sigma_{p,T} \\
&= \sum_{p \in Z_k(f)} (s, \sigma_{p,T})(\sigma_{p,T} - \pi_T(c) \sigma_{p,T}) - \sum_{p \in Z_k(f)} (s, \sigma_{p,T} - \pi_T(c) \sigma_{p,T}) \pi_T(c) \sigma_{p,T}
\end{align*}
Now by applying the triangle inequality, and Cauchy-Schwartz to the inequality and then using Lemma~\ref{lem:58} we see that for $T$ sufficiently large, there is a constant $C_2$ such that
$$ \| \pi_T s \| \leq \frac{C_2}{T} \| s\|,$$
which implies that, for $T$ sufficiently large
$$\| \pi_T^\perp s \| = \| s - \pi_T s\| \geq \|s\| - \| \pi_T s \| \geq C_3 \|s\|,$$
for some constant $C_3$. We therefore find, using Proposition~\ref{prop:est}
$$ C T C^2_3 \|s\|^2 \leq C T \|\pi_T^\perp s\|^2 \leq (\pi_T^\perp s, \mathcal{D}_T \pi_T^\perp s).$$
Now, considering the original definition in Proposition~\ref{prop:D} we may write 
$$\mathcal{D}_T = \sum_{i=1}^5 K_i^\dagger K_i,$$
which allows us to write, using Proposition~\ref{prop:est}
\begin{align*} (\pi_T^\perp s, \mathcal{D}_T \pi_T^\perp s) &= \sum_{i=1}^5 \| K_i\pi_T^\perp s \|^2 =  \sum_{i=1}^5 \| K_i(s-\pi_T s)\|^2 \leq 2 \sum_{i=1}^5( \| K_is \|^2 + \| K_i \pi_T s\|^2) \\ &=  (s, \mathcal{D}_T  s) + (\pi_T s, D_T \pi_T s) \leq (s, \mathcal{D}_T s) + \frac{1}{T^2} \|s\|^2.
\end{align*}
We then find
$$ CTC_3^2 \|s\|^2 \leq (s, \mathcal{D}_T s)+\frac{1}{T^2} \|s \|^2,$$
rearranging gives
$$ (s, \mathcal{D}_T s) \geq  CTC_3^2 \|s\|^2-\frac{1}{T^2} \|s \|^2.$$
But if $\|s\| \neq 0$, then as $T \to \infty$ this contradicts the initial assumption that $s \in E_T(c)$. Hence $\|s\|=0$ and we have
$$ {\rm dim}\, E_T(c)  = {\rm dim} \, E_T.$$
\end{proof}

\newpage
\section{Proof of the inequalities} \label{sec:proof}
The dimension of the kernel of $D_{PT}$ must be less than or equal to the number eigenvalues in the interval $[0,c]$. From Proposition~\ref{prop:vals} we have the following.
\begin{cor}
$${\rm dim}\, {\rm ker}\, D_{PT} = {\rm dim} PH^k_{d+d^\Lambda}(M) \leq m_k$$
\end{cor}
Applying the Lefschetz decomposition then proves Theorem~\ref{thm:1} for the $d+d^\Lambda$ cohomology. It remains to establish the equivalent result for the $dd^\Lambda$ cohomology. This is done by appealing to Hodge duality. 
\begin{lem}
The Hodge star operator acts as
$$\ast : \mathcal{H}^k_{d_f+ d_f^\Lambda}(M) \to \mathcal{H}^{2n-k}_{d_{-f} d^\Lambda_{-f}}(M)$$
\end{lem}
\begin{proof}
As discussed in Ref.\cite{tseng2012cohomology}, the harmonic $dd^\Lambda$ forms are given by solving
$$ dd^\Lambda \alpha = 0, \quad d^\ast \alpha=0, \quad d^{\Lambda \ast} \alpha = 0.$$
The result then follows from the properties of the differentials given in Sections~\ref{sec:hodge} and~\ref{sec:witten}.
\end{proof}
Applying the Lefschetz decomposition to the harmonic forms ${H}^{2n-k}_{d_{-f} d^\Lambda_{-f}}$, and observing that $f \to -f$ sends a Morse index $n_p$ to $2n-n_p$ then completes the proof of Theorem~\ref{thm:1}.

\bibliographystyle{plain}

\begin{thebibliography}{10}

\bibitem{angella2017quantitative}
Daniele Angella and Nicoletta Tardini.
\newblock Quantitative and qualitative cohomological properties for
  non-{K}{\"a}hler manifolds.
\newblock {\em Proc. Amer. Math. Soc.}, 145(1):273--285, 2017.

\bibitem{angella2014symplectic}
Daniele Angella and Adriano Tomassini.
\newblock Symplectic manifolds and cohomological decomposition.
\newblock {\em J. Symplectic Geom.}, 12(2):215--236, 2014.

\bibitem{angella2015inequalities}
Daniele Angella and Adriano Tomassini.
\newblock Inequalities {\`a} la {F}r{\"o}licher and cohomological
  decompositions.
\newblock {\em J. Noncommutative Geom.}, 9(2):505--542, 2015.

\bibitem{bahramgiri}
Mohsen Bahramgiri.
\newblock Symplectic {H}odge theory, harmonicity, and {T}hom duality.
\newblock {\em Iran. J. Sci. Technol. Trans. A: Sci.}, 37(3):359--363, 2013.

\bibitem{bismut1991complex}
Jean-Michel Bismut and Gilles Lebeau.
\newblock Complex immersions and {Q}uillen metrics.
\newblock {\em Publ. Math. Inst. Hautes {\'E}tudes Sci.}, 74(1):1--291, 1991.

\bibitem{bouche1990cohomologie}
Thierry Bouche.
\newblock La cohomologie coeffective d'une vari{\'e}t{\'e} symplectique.
\newblock {\em Bull. Sci. Math}, 114(2):115--122, 1990.

\bibitem{brylinski1988differential}
Jean-Luc Brylinski.
\newblock A differential complex for {P}oisson manifolds.
\newblock {\em J. Differential Geom.}, 28(1):93--114, 1988.

\bibitem{cieliebak2012stein}
Kai Cieliebak and Yakov Eliashberg.
\newblock {\em From Stein to Weinstein and back: symplectic geometry of affine
  complex manifolds}, volume~59.
\newblock American Mathematical Soc., 2012.

\bibitem{douglas2012banach}
Ronald~G Douglas.
\newblock {\em Banach algebra techniques in operator theory}.
\newblock Springer Science \& Business Media, 1998.

\bibitem{evens1999poisson}
Sam Evens and Jiang-Hua Lu.
\newblock {P}oisson harmonic forms, {K}ostant harmonic forms, and the
  ${S}^1$-equivariant cohomology of ${K}/{T}$.
\newblock {\em Adv. Math.}, 142(2):171--220, 1999.

\bibitem{floer1989witten}
Andreas Floer.
\newblock Witten's complex and infinite-dimensional {M}orse theory.
\newblock {\em J. Differential Geom.}, 30(1):207--221, 1989.

\bibitem{fukaya1999arnold}
Kenji Fukaya and Kaoru Ono.
\newblock Arnold conjecture and {G}romov--{W}itten invariant.
\newblock {\em Topology}, 38(5):933--1048, 1999.

\bibitem{golovko2020variants}
Roman Golovko.
\newblock On variants of {A}rnold conjecture.
\newblock {\em Arch. Math.}, 56(5):277--286, 2020.

\bibitem{kodaira201548}
Kunihiko Kodaira and Donald~Clayton Spencer.
\newblock {\em On deformations of complex analytic structures, III, Stability
  theorems for complex structures}.
\newblock Princeton University Press, 2015.

\bibitem{koszul1985crochet}
Jean-Louis Koszul.
\newblock Crochet de {S}chouten-{N}ijenhuis et cohomologie.
\newblock {\em Ast{\'e}risque}, 137:257--271, 1985.

\bibitem{lin2011currents}
Yi~Lin.
\newblock Currents, primitive cohomology classes and symplectic hodge theory.
\newblock {\em arXiv preprint arXiv:1112.2442}, 2011.

\bibitem{liu1998floer}
Gang Liu and Gang Tian.
\newblock Floer homology and {A}rnold conjecture.
\newblock {\em J. Differential Geom.}, 49(1):1--74, 1998.

\bibitem{machon2020poisson}
Thomas Machon.
\newblock A {P}oisson bracket on the space of {P}oisson structures.
\newblock {\em arXiv preprint arXiv:2008.11074}, 2020.

\bibitem{mcmullen19994}
Curtis~T McMullen and Clifford~H Taubes.
\newblock 4-manifolds with inequivalent symplectic forms and 3-manifolds with
  inequivalent fibrations.
\newblock {\em Math. Res. Lett.}, 1999.

\bibitem{novikov1982hamiltonian}
Sergei~Petrovich Novikov.
\newblock The {H}amiltonian formalism and a many-valued analogue of {M}orse
  theory.
\newblock {\em Uspekhi Mat. Nauk}, 37(5):3--49, 1982.

\bibitem{palais1959natural}
Richard~S Palais.
\newblock Natural operations on differential forms.
\newblock {\em Trans. Amer. Math. Soc.}, 92(1):125--141, 1959.

\bibitem{schweitzer2007autour}
Michel Schweitzer.
\newblock Autour de la cohomologie de {B}ott-{C}hern.
\newblock {\em arXiv preprint arXiv:0709.3528}, 2007.

\bibitem{smith1976examples}
Roland~T Smith.
\newblock Examples of elliptic complexes.
\newblock {\em Bull. Am. Math. Soc.}, 82(2):297--299, 1976.

\bibitem{tanaka2018odd}
Hiro~Lee Tanaka and Li-Sheng Tseng.
\newblock Odd sphere bundles, symplectic manifolds, and their intersection
  theory.
\newblock {\em Camb. J. Math.}, 6(3):213--266, 2018.

\bibitem{tardini2017cohomological}
Nicoletta Tardini.
\newblock Cohomological aspects on complex and symplectic manifolds.
\newblock In {\em Complex and Symplectic Geometry}, pages 231--247. Springer,
  2017.

\bibitem{tardini2019symplectic}
Nicoletta Tardini and Adriano Tomassini.
\newblock Symplectic cohomologies and deformations.
\newblock {\em Bolletino dell Unione Mat. Ital.}, 12(1):221--237, 2019.

\bibitem{tsai2016cohomology}
Chung-Jun Tsai, Li-Sheng Tseng, and Shing-Tung Yau.
\newblock Cohomology and hodge theory on symplectic manifolds: {III}.
\newblock {\em J. Differential Geom.}, 103(1):83--143, 2016.

\bibitem{tseng2012cohomology}
Li-Sheng Tseng and Shing-Tung Yau.
\newblock Cohomology and {H}odge theory on symplectic manifolds: I.
\newblock {\em J. Differential Geom.}, 91(3):383--416, 2012.

\bibitem{tseng2012cohomology2}
Li-Sheng Tseng and Shing-Tung Yau.
\newblock Cohomology and hodge theory on symplectic manifolds: {II}.
\newblock {\em J. Differential Geom.}, 91(3):417--443, 2012.

\bibitem{weil1958introduction}
Andr{\'e} Weil.
\newblock {\em Introduction {\`a} l'{\'e}tude des vari{\'e}t{\'e}s
  {k}{\"a}hl{\'e}riennes}, volume 1267.
\newblock Hermann, 1958.

\bibitem{weinstein1997modular}
Alan Weinstein.
\newblock The modular automorphism group of a {P}oisson manifold.
\newblock {\em J. Geom. Phys.}, 23(3-4):379--394, 1997.

\bibitem{witten1982supersymmetry}
Edward Witten.
\newblock Supersymmetry and {M}orse theory.
\newblock {\em J. Differential Geom.}, 17(4):661--692, 1982.

\bibitem{yan1996hodge}
Dong Yan.
\newblock Hodge structure on symplectic manifolds.
\newblock {\em Adv. Math.}, 120(1):143--154, 1996.

\bibitem{zhang2001lectures}
Weiping Zhang.
\newblock {\em Lectures on Chern-Weil theory and Witten deformations}, volume~4
  of {\em Nankai Tracts in Mathematics}.
\newblock World Scientific, 2001.

\end{thebibliography}

\end{document}